\documentclass[preprint,11pt]{article}
\usepackage{amsmath}
\usepackage{amsthm}
\usepackage{amsfonts}
\usepackage[latin5]{inputenc}
\usepackage{color}

\begin{document}
\numberwithin{equation}{section}
\newcounter{thmcounter}
\newcounter{Remarkcounter}
\newcounter{Defcounter}
\numberwithin{thmcounter}{section}
\newtheorem{Prop}[thmcounter]{Proposition}
\newtheorem{Corol}[thmcounter]{Corollary}
\newtheorem{theorem}[thmcounter]{Theorem}
\newtheorem{Lemma}[thmcounter]{Lemma}
\theoremstyle{definition}
\newtheorem{Def}[Defcounter]{Definition}
\theoremstyle{remark}
\newtheorem{Remark}[Remarkcounter]{Remark}
\newtheorem{Example}[Remarkcounter]{Example}

\newcommand{\COD}[3]{Codazzi equation \eqref{MinkCodazzi} for $X=e_{#1}$, $Y=e_{#2}$ and $Z=e_{#3}$}
\newcommand{\diag}{\mathrm{diag}\ }
\newcommand{\DIV}{\mathrm{div}}
\title{Some classifications of biharmonic Lorentzian hypersurfaces  in Minkowski 5-space $\mathbb E^5_1$}
\author{Nurettin Cenk Turgay \footnote{The final version of this paper is going to  appear in Mediterranean Journal of Mathematics (see \cite{TurgayMink5Bih})}\footnote{Istanbul Technical University, Faculty of Science and Letters, Department of  Mathematics, 34469 Maslak, Istanbul, Turkey} \footnote{e-mail:turgayn@itu.edu.tr} \footnote{This work is supported by Scientific Research Agency of Istanbul Technical University.}}
\date{}

\maketitle
\begin{abstract}
In this paper, we study Lorentzian hypersurfaces in Minkowski 5-space with non-diagonalizable shape operator whose characteristic polinomial is $(t-k_1)^2(t-k_3)(t-k_4)$ or $(t-k_1)^3(t-k_4)$. We proved that  in these cases, a hypersurface is biharmonic if and only if it is minimal.


\end{abstract}

\section{Introduction}\label{SectionIntrod}
Let  $M$ be an $n$-dimensional submanifold of a semi-Euclidean space $\mathbb E^m_s$ and $x:M\rightarrow\mathbb E^m_s$ an isometric immersion. $M$ is said to be biharmonic if $x$ satisfies $\Delta^2 x=0$, or, equivalently, $\Delta \vec{H} =\lambda \vec{H}$ where $\Delta$ and $\vec{H}$  are the Laplace operator and  mean curvature vector of $M$, respectively. Biharmonic hypersurfaces are studied by many geometers, after it is conjectured that every a submanifold of a Euclidean space is minimal by Bang-Yen Chen (see \cite{ChenOpenProblems,ChenRapor}).

Note that, there are some results on hypersurfaces of Euclidean spaces which provide affirmative partial solutions to Chen's original biharmonic conjecture, \cite{ChenKitap,Defever1996,HsanisField,Jiang1985}. For example, Yu Fu has studied biharmonic hypersurfaces in $\mathbb E^5$ with at most 3 principle curvatures and he has proved that the conjecture is true for this case, \cite{YuFu2014}. 

On the other hand, in semi-Euclidean spaces, there are non-minimal (or non-maximal) biharmonic submanifolds. In other words, Chen's conjecture is not valid if the ambient space is semi-Euclidean. For example, some non-minimal biharmonic surfaces in $\mathbb E^4_1$ and $\mathbb E^4_2$ were obtained in \cite{ChenIshikawa1991Bih} and \cite{ChenIshikawa1998Bih}. In particular, some results on biharmonic submanifolds of semi-Euclidean spaces have been appeared recently, \cite{Arvan1,Arvan2,Arvan3,Defever2006,Fu2013MinkBih,Papantoniou}. For example, in \cite{Arvan1} Arvanitoyeorgos et al.  proved that all biharmonic Lorentzian hypersurfaces in Minkowski 4-space are minimal. Futher,  they also proved that a biharmonic hypersurface in $\mathbb E^4_s$  has constant mean curvature in \cite{Arvan2}. In addition, in \cite{Papantoniou}, Papantoniou et al. proved that a nondegenerate biharmonic hypersurface with index 2 in  $\mathbb E^4_2$ is minimal.

In this work we study biharmonic Lorentzian hypersurfaces in Minkowski 5-space with non-diagonalizable shape operator with at most 3 distinct eigenvalues.  After  we give basic definitions and notation in Section 2, we obtain our main results in Section 3.

The hypersurfaces we are dealing with are smooth and connected unless otherwise stated.

\section{Prelimineries}

\subsection{Basic notation, formulas and definitions}
Let $\mathbb E^m_s$ denote the pseudo-Euclidean $m$-space with the canonical 
pseudo-Euclidean metric tensor $g$ of index $s$ given by  
$$
 g=-\sum\limits_{i=1}^s dx_i^2+\sum\limits_{j=s+1}^m dx_j^2,
$$
where $(x_1, x_2, \hdots, x_m)$  is a rectangular coordinate system in $\mathbb E^m_s$. A non-zero vector $v\in T_p(\mathbb E^m_s)\cong \mathbb E^m_s$  is called space-like (resp. time-like or light-like) if $\langle v,v\rangle>0$  (resp. $\langle v,v\rangle<0$ or $\langle v,v\rangle=0$), where $T_p(\mathbb E^m_s)$ $\langle\ ,\ \rangle$ denotes the indefinite inner product of $\mathbb E^m_s$.

Consider an $n$-dimensional immersed semi-Riemannian submanifold  $M^n_r$  of the space $\mathbb E^m_s$. We denote Levi-Civita connections of $\mathbb E^m_s$ and $M$ by $\widetilde{\nabla}$ and $\nabla$, respectively. Then, the Gauss and Weingarten formulas are given, respectively, by
\begin{eqnarray}
\label{MEtomGauss} \widetilde\nabla_X Y&=& \nabla_X Y + h(X,Y),\\
\label{MEtomWeingarten} \widetilde\nabla_X \zeta&=& -A_\xi(X)+\nabla^\perp_X \zeta
\end{eqnarray}
for all tangent vectors fields $X,\ Y$ and normal vector fields $\zeta$,  where $h$,  $\nabla^\perp$  and  $A$ are the second fundamental form, the normal connection and  the shape operator of $M$, respectively. Note that shape operator and the second fundamental form are related by  $\left\langle h(X, Y), \zeta \right\rangle = \left\langle A_{\zeta}X, Y \right\rangle.$
 
The Gauss and Codazzi equations are given, respectively, by
\begin{eqnarray}
\label{MinkGaussEquation}\label{GaussEq} \langle R(X,Y,)Z,W\rangle&=&\langle h(Y,Z),h(X,W)\rangle-
\langle h(X,Z),h(Y,W)\rangle,\\
\label{MinkCodazzi} (\bar \nabla_X h )(Y,Z)&=&(\bar \nabla_Y h )(X,Z),
\end{eqnarray}
where  $R$ is the curvature tensor associated with connection $\nabla$ and  $\bar \nabla h$ is defined by
$$(\bar \nabla_X h)(Y,Z)=\nabla^\perp_X h(Y,Z)-h(\nabla_X Y,Z)-h(Y,\nabla_X Z).$$

\subsection{Lorentzian hypersurfaces in $\mathbb E^5_1$}
Let $M$ be an oriented Lorentzian hypersurface in $\mathbb E^5_1$ with non-diagonalizable shape operator $S$ with at most 3 distinct eigenvalues. It is well-known that the matrix represantation of $S$ with respect to a pseudo-orthonormal appropriate frame field $\{e_1,e_2,e_3,e_4\}$ is in one of the following two forms.
\begin{align} \label{SOPCASES}
\mbox{Case I. }
S=\left( \begin{array}{cccc}
k_1 &1&0&0\\
0 &k_1&0&0\\
0 &0&k_3&0\\
0 &0&0&k_4\\
\end{array}\right),&\quad
\mbox{Case II. }S=\left( \begin{array}{cccc}
k_1 &0&0&0\\
0 &k_1&1&0\\
-1 &0&k_1&0\\
0 &0&0&k_4\\
\end{array}
\right)
\end{align}
for some smooth functions $k_1,k_3$ and $k_4$, (see \cite{Lucas2011, ONeillKitap}). With the abuse of terminology, we call these vector fields $e_1,e_2,e_3,e_4$ as principal directions and the function $s_1=\mathrm{tr} S$ as the (first) mean curvature of $M$. Note that $M$ is said to be (1-) minimal if and only if $s_1=0.$

Note that for a pseudo-orthonormal frame field $\{e_1,e_2,e_3,e_4\}$ satisfying
$$\langle e_A,e_B\rangle=1-\delta_{AB},\quad \langle e_A,e_a\rangle=0,\quad \langle e_a,e_b\rangle=\delta_{ab}$$
for all $A,B=1,2$, $a,b=3,4$, the induced connection $\nabla$ of $M$ becomes
\begin{subequations}\label{Cse1ConFormALL}
\begin{eqnarray}
\label{Cse1ConForm1} \nabla_{e_i}e_1&=& \phi_i e_1+\omega_{13}(e_i)e_3+\omega_{14}(e_i)e_4,\\
\label{Cse1ConForm2} \nabla_{e_i}e_2&=& -\phi_i e_2+\omega_{23}(e_i)e_3+\omega_{24}(e_i)e_4,\\
\label{Cse1ConForm3} \nabla_{e_i}e_3&=& \omega_{23}(e_i)e_1+\omega_{13}(e_i)e_2+\omega_{34}(e_i)e_4,\\
\label{Cse1ConForm4} \nabla_{e_i}e_4&=& \omega_{24}(e_i)e_1+\omega_{14}(e_i)e_2-\omega_{34}(e_i)e_3,
\end{eqnarray}
\end{subequations}
where $\phi_i=\phi (e_i)=\langle \nabla_{e_i}e_2,e_1\rangle$ and $\omega_{jk}(e_i)=\langle\nabla_{e_i}e_j,e_k\rangle$, i.e., $\phi=-\omega_{12}$. Laplace operator $\Delta$ is
$$\Delta=e_1e_2+e_2e_1-e_3e_3-e_4e_4-\nabla_{e_1}e_2-\nabla_{e_2}e_1+\nabla_{e_3}e_3+\nabla_{e_4}e_4.$$

From \cite[p. 165]{Lucas2011}, we see that $M$ is biharmonic if and only if 
\begin{subequations}
\begin{eqnarray}
\label{HSurfCondE51Case1e1} S(\nabla s_1)+\frac{s_1}2\nabla s_1&=&0,\\
\label{HSurfCondE51Case1e2} \Delta s_1+s_1\mathrm{tr} S^2&=&0.
\end{eqnarray}
\end{subequations}
Note that if \eqref{HSurfCondE51Case1e1} is satisfied, then $M$ is said to be an $\mbox{H}$-hypresurface, \cite{HsanisField} or a bi-conservative hypersurface, \cite{Fu2013MinkBih}. Moreover, we call \eqref{HSurfCondE51Case1e2} as biharmonic equation.

\section{Biharmonic Lorentzian hypersurfaces}
In this section we focus on Lorentzian hypersurfaces with the shape operator given in case I and case II of \eqref{SOPCASES} seperately.
\subsection{Case I}
We consider the case I in \eqref{SOPCASES}, i.e., the shape operator is
\begin{equation}\label{Case1ShapeOperator}
Se_1=k_1e_1,\quad Se_2=e_1+k_1e_2,\quad Se_3=k_3e_3,\quad Se_4=k_4e_4
\end{equation}
with the characteristic polinomial $(k-k_1)^2(k-k_3)(k-k_4)$.  In this case we have $s_1=2k_1+k_3+k_4$ and $\mathrm{tr} S^2=2k_1^2+k_3^2+k_4^2$. We also assume that the functions $k_1-k_3$, $k_1-k_4$ and $k_3-k_4$ does not vanish on $M$.

Now, assume that $M$ is biharmonic. If $s_1$ is constant then \eqref{HSurfCondE51Case1e1} implies $s_1=0$, i.e., $M$ is minimal, \cite{Arvan1}. Thus, we assume $\nabla s_1$ is non-vanishing on $M$. Then, \eqref{HSurfCondE51Case1e1} implies that $\nabla s_1$ is not only a principal direction but also an eigenvector of $S$. Thus, without loss of generality, we may assume either $e_4=\frac{\nabla s_1}{\|\nabla s_1\|}$ or $e_1=\frac{\nabla s_1}{\|\nabla s_1\|}$. If $e_1=\frac{\nabla s_1}{\|\nabla s_1\|}$, then we have $e_2(s_1)=e_3(s_1)=e_4(s_1)=0$. Moreover, Codazzi equation \eqref{MinkCodazzi} for $X=e_1,Y=Z=e_2$ gives $e_1(s_1)=0$ which implies $s_1$ is constant which yields a contradiction. Hence, we have
\begin{subequations}\label{Case1ShapeOperatorResAll}
\begin{equation}\label{Case1ShapeOperatorRes1}
e_4=\frac{\nabla s_1}{\|\nabla s_1\|}, \quad k_4=-\frac{s_1}{2},\quad 2k_1+k_3=\frac32 s_1
\end{equation}
which imply
\begin{equation}\label{Case1ShapeOperatorRes1a}
e_4(k_4)\neq0,\quad e_A(k_4)=0,\quad A=1,2,3.
\end{equation}
\end{subequations}
On the other hand, biharmonic equation \eqref{HSurfCondE51Case1e2} becomes
\begin{equation}\label{HSurfCondE51Case1e2v2}
e_4e_4(k_4)+\left(\omega_{24}(e_1)+\omega_{14}(e_2)-\omega_{34}(e_3)\right)e_4(k_4)=k_4(2k_1^2+k_3^2+k_4^2).
\end{equation}

From \eqref{Case1ShapeOperator} we see that the only non-zero terms of second fundamental form of $M$ are
\begin{align}\label{Case1ShapeOperatorSec}
\begin{split}
 h(e_1,e_2)=-k_1N,&\quad  h(e_2,e_2)=-N,\quad  h(e_3,e_3)=k_3N,\quad h(e_4,e_4)=k_4N,
\end{split}
\end{align}
where $N$ is unit the normal vector field associated with the orientation of $M$. 

We apply the Codazzi equation \eqref{MinkCodazzi} for $X=e_i,$ $Y=e_j$ and $Z=e_k$ for each triplet (i, j, k) in the set
$\{(1,4,4),(2,4,4),(3,4,4), (4,3,3), (2,1,1), (1,3,3),(3,1,2), (1,3,2), (3,1,1)\}$
and combine equations obtained with \eqref{Case1ShapeOperatorResAll}, \eqref{Case1ShapeOperatorSec} to  get
\begin{subequations}
\begin{eqnarray}
\label{Cse1ConFormEq1} \omega_{A4}(e_4)&=&0,\quad A=1,2,3,\\
\label{Cse1ConFormEq1b} e_4(k_3)&=&\omega_{34}(e_3)(k_3-k_4),\\
\label{Cse1ConFormEq2a} e_1(k_1)=e_1(k_3)&=&0,\\
\label{Cse1ConFormEq2b} \label{Cse1ConFormEq7a}\omega_{13}(e_1)=\omega_{13}(e_3)&=&0,\\
\label{Cse1ConFormEq7b}\omega_{23}(e_1)=\omega_{13}(e_2)&=& \frac{e_3(k_1)}{k_3-k_1}.
\end{eqnarray}
\end{subequations}

Note that for a tangent vector field $X$ on $M$,  $\langle X,e_4\rangle=0$ if and only if $Xk_4=0$. Thus, we have $\langle[e_1,e_2],e_4\rangle=0$ which implies 
\begin{subequations}\label{Cse1ConFormEq4ALL} 
\begin{equation}\label{Cse1ConFormEq4a}
\omega_{24}(e_1)=\omega_{14}(e_2)
\end{equation}
Moreover, by computing \COD 124 and for $X=e_1$, $Y=e_4$, $Z=e_2$ we obtain
\begin{eqnarray}
\label{Cse1ConFormEq4b}e_4(k_1)&=&\omega_{24}(e_1)(k_4-k_1),\\
\label{Cse1ConFormEq4c}\omega_{14}(e_1)&=&0.
\end{eqnarray}
\end{subequations}
Similarly, we have $\langle[e_1,e_3],e_4\rangle=\langle[e_2,e_3],e_4\rangle=0$. By combaining these equations with \eqref{MinkCodazzi}, we obtain
\begin{subequations}\label{Cse1ConFormEq5ALL} 
\begin{eqnarray}
\label{Cse1ConFormEq5a}\omega_{34}(e_1)=\omega_{14}(e_3)=\omega_{13}(e_4)&=&0,\\
\label{Cse1ConFormEq5b}\omega_{34}(e_2)=\omega_{24}(e_3)=\omega_{23}(e_4)&=&0.
\end{eqnarray}
\end{subequations}

On the other hand, because of \eqref{Case1ShapeOperatorRes1a}, we have $[e_A,e_4](k_4)=e_Ae_4(k_4)$. We compute the left hand side of this equation  by using \eqref{Case1ShapeOperatorRes1a} and \eqref{Cse1ConFormEq1} and obtain
\begin{equation}
\label{Cse1ConFormEq6}e_Ae_4(k_4)=e_Ae_4e_4(k_4)=0.
\end{equation}
Now, we want to show 
\begin{equation}
\label{Cse1ConFormClaim1}e_3(k_1)=e_3(k_3)=0.
\end{equation}
By combaining \eqref{Cse1ConFormEq4a}, \eqref{Cse1ConFormEq4b} and \eqref{Cse1ConFormEq7b} with the Gauss equation $R(e_3,e_2,e_4,e_1)=0$, we obtain
\begin{equation}
\label{Cse1ConFormClaim1Eq1}e_3(\omega_{14}(e_2))=\frac{e_3(k_1)}{k_1-k_3}\left(\omega_{14}(e_2)+\omega_{34}(e_3)\right)
\end{equation}
From which and \eqref{Cse1ConFormEq4b} we have
\begin{equation}
\label{Cse1ConFormClaim1Eq2}
e_3e_4(k_1)=\frac{e_3(k_1)}{k_1-k_3}\left(\omega_{14}(e_2)(k_3+k_4-2k_1)+\omega_{34}(e_3)(k_4-k_1)\right).
\end{equation}
Note that \eqref{Cse1ConFormEq1b} implies 
$$e_3(\omega_{34}(e_3))=e_3\left(\frac{e_4(k_3)}{k_3-k_4}\right).$$
By taking into account $e_3(2k_1+k_3)=0$, we use \eqref{Cse1ConFormClaim1Eq2} to compute the right-hand side of the above equation and we get
\begin{equation}
\label{Cse1ConFormClaim1Eq3}e_3(\omega_{34}(e_3))=\frac{2(k_3+k_4-2k_1)e_3(k_1)}{(k_3-k_4)(k_3-k_1)}\left(\omega_{14}(e_2)+\omega_{34}(e_3)\right).
\end{equation}
On the other hand, by applying $e_3$ to \eqref{HSurfCondE51Case1e2v2} and using \eqref{Cse1ConFormEq6} we have 
\begin{equation}\label{HSurfCondE51Case1e2v2e3applied}
e_3\left(\omega_{24}(e_1)+\omega_{14}(e_2)-\omega_{34}(e_3)\right)e_4(k_4)=k_4e_3(2k_1^2+k_3^2)
\end{equation}
By combaining \eqref{Cse1ConFormClaim1Eq1}, \eqref{Cse1ConFormClaim1Eq3} and \eqref{HSurfCondE51Case1e2v2e3applied}, we obtain
\begin{equation}
\label{Cse1ConFormClaim1Eq4}\frac{e_3(k_1)}{(k_3-k_4)}\left(\omega_{14}(e_2)+\omega_{34}(e_3)\right)e_4(k_4)=e_3(k_1)(k_1-k_3)k_4.
\end{equation}
which gives 
\begin{equation}
\label{Cse1ConFormClaim1Eq4b}\frac{\omega_{14}(e_2)+\omega_{34}(e_3)}{(k_3-k_4)(k_1-k_3)}=\frac {k_4}{e_4(k_4)}.
\end{equation}
on the open subset $\mathcal O=\{m\in M|e_3(k_1)|_m\neq0\}$ of $M$. By applying $e_3$ to this equation and using \eqref{Cse1ConFormClaim1Eq1}, \eqref{Cse1ConFormClaim1Eq3} we get $k_1=k_3$ on $\mathcal O$ which is a contradiction unless $\mathcal O$  is not empty. Hence, we have proved \eqref{Cse1ConFormClaim1}. Therefore, \eqref{Cse1ConFormEq7b}, \eqref{Cse1ConFormClaim1Eq1} give
\begin{equation}
\label{Cse1ConFormClaim1Res1a}\omega_{23}(e_1)=\omega_{13}(e_2)=0.
\end{equation}

By a similar way, we obtained
\begin{equation}
\label{Cse1ConFormClaim2}e_2(k_1)=e_2(k_3)=0.
\end{equation}
Moreover, by combaining this equation with \COD 233 we get
\begin{equation}
\label{Cse1ConFormClaim1Res2a}\omega_{23}(e_3)=0.
\end{equation}

By summing up the equations obtained sofar and using Gauss equation \eqref{MinkGaussEquation}, we obtain the following lemma.
\begin{Lemma}\label{Lemmmma1}
Let $M$ be a Lorentzian hypersurface in $\mathbb E^5_1$ with the shape operator given by \eqref{Case1ShapeOperator}. Then the funnctions $\xi=\omega_{14}(e_2)$ and $\eta={\omega_{34}(e_3)}$  satisfy
\begin{subequations}\label{Cse1ConFormClaim2Lemma1Eq1All}
\begin{eqnarray}
\label{Cse1ConFormClaim2Lemma1Eq1a}e_4(\xi)&=&-\xi^2-k_1k_4,\\
\label{Cse1ConFormClaim2Lemma1Eq1b}e_4(\eta)&=&\eta^2+k_3k_4,\\
\label{Cse1ConFormClaim2Lemma1Eq1c}\xi\eta&=&k_1k_3.
\end{eqnarray}
\end{subequations}
\end{Lemma}

Now, we are ready to prove the following classification theorem
\begin{theorem}
Let $M$ be a Lorentzian hypersurface in $\mathbb E^5_1$ with the shape operator given by \eqref{Case1ShapeOperator}. Then $M$ is biharmonic if and only if it is minimal.
\end{theorem}

\begin{proof}
First, we want to prove the necessary condition. Let $M$ be a non-minimal biharmonic hypersurface. Note that if $M$ has constant mean curvature, then biharmonic equation \eqref{HSurfCondE51Case1e2v2} implies $s_1=0$. Thus, we have $k_1=-s_1/2$ is non-constant. Because of Lemma \ref{Lemmmma1}, \eqref{Cse1ConFormClaim2Lemma1Eq1All} is satisfied.

By combaining \eqref{Cse1ConFormClaim2Lemma1Eq1a} and \eqref{Cse1ConFormClaim2Lemma1Eq1b} with  \eqref{Cse1ConFormEq1b} and \eqref{Cse1ConFormEq4b} we obtain
\begin{eqnarray}
\nonumber e_4e_4(k_1)&=& \xi e_4(k_4)+\left(k_1-k_4\right) \left(k_1 k_4+2 \xi ^2\right),\\
\nonumber e_4e_4(k_3)&=& -\eta e_4(k_4)+\left(k_3-k_4\right) \left(k_3 k_4+2 \eta ^2\right)
\end{eqnarray}
from which and \eqref{Case1ShapeOperatorRes1} we get 
\begin{align}
\label{LastEquations1} e_4e_4(k_4)=  \frac{1}{3}e_4(k_4) (\eta -2 \xi )+\frac{1}{3} \left(-\left(k_3-k_4\right) \left(k_3 k_4+2 \eta ^2\right)-2 \left(k_1-k_4\right) \left(k_1 k_4+2 \xi ^2\right)\right).
\end{align}
From \eqref{Cse1ConFormEq1b}, \eqref{Cse1ConFormEq4b} and \eqref{Case1ShapeOperatorRes1} we also have 
\begin{align}
\label{LastEquations2} 3e_4(k_4)=-2 \left(k_4-k_1\right) \xi -\left(k_3-k_4\right) \eta .
\end{align}

By combaining \eqref{Case1ShapeOperatorRes1} and \eqref{LastEquations1} with biharmonic equation \eqref{HSurfCondE51Case1e2v2}, we get
\begin{align}
\label{LastEquations3}e_4(k_4) (2 \eta -\xi )=2 k_4 \left(12 k_1^2+21 k_1 k_4+10 k_4^2\right).
\end{align}
By applying $e_4$ on this equation and using \eqref{Cse1ConFormClaim2Lemma1Eq1All}, \eqref{LastEquations2} and \eqref{LastEquations3}, we obtain
\begin{align}
\label{LastEquations4} 3 \left(k_1+k_4\right) \left(5 k_1+11 k_4\right)e_4(k_4)= \left(30 k_4^3+54 k_1 k_4^2+30 k_1^2 k_4\right) \eta - \left(56 k_4^3+72 k_1 k_4^2+15 k_1^2 k_4\right) \xi.
\end{align}

By combaining \eqref{LastEquations2}-\eqref{LastEquations5} we get
\begin{eqnarray}
\label{LastEquations5} \left(2 k_1+4 k_4\right) \eta ^2+\left(k_4-k_1\right) \xi ^2=18 k_4^3+58 k_1 k_4^2+38 k_1^2 k_4,\\
\label{LastEquations6} \left(20 k_1^3+44  k_1^2k_4+64 k_1 k_4^2 +28 k_4^3\right) \eta +\left(20 k_1^3+74 k_1^2k_4 +124  k_1k_4^2+68 k_4^3\right) \xi =0.
\end{eqnarray}
Next, by using \eqref{Cse1ConFormClaim2Lemma1Eq1c}, we eleminate $\xi$ and $\eta$ from these equations and obtain $k_1=a k_4$ for a constant $a$ satisfying the nineth degree polynomial
$$100 a^9-210 a^8-4306 a^7-19687 a^6-49256 a^5-79972 a^4-86866 a^3-60384 a^2-24178 a-42840=0.$$
From \eqref{Case1ShapeOperatorRes1} we also have  $k_3=(-2a-3) k_4$. Therefore, \eqref{Cse1ConFormEq1b}, \eqref{Cse1ConFormEq4b} and \eqref{Cse1ConFormClaim2Lemma1Eq1All} imply
\begin{eqnarray}\label{LastEquations7}
 &e_4(k_4)=\frac{2a+4}{2a+3}\eta k_4,\quad  e_4(k_4)=\frac{(1-a)}a\xi k_4,&\\
 &e_4(\eta)=\eta^2+(-2a-3)k_4^2,\quad e_4(\xi)=-\xi^2-ak_4^2, \quad \xi\eta=(-2a-3)ak_4^2.&
\end{eqnarray}
A direct calculation shows that these equations imply $a=-1$ in which case we have $k_1=k_3$. However, this is a contradiction. Hence, we proved the necessary condition of the theorem.

The converse is obvious.
\end{proof}


\subsection{Case II}
Now, we consider the case III in \eqref{SOPCASES}, i.e.,  the shape operator is
\begin{equation}\label{Case2ShapeOperator}
Se_1=k_1e_1-e_3,\quad Se_2=k_1e_2,\quad Se_3=e_2+k_1e_3,\quad Se_4=k_4e_4
\end{equation}
with the characteristic polinomial $(k-k_1)^3(k-k_4)$. In this case we have $s_1=3k_1+k_4$ and $\mathrm{tr} S^2=3k_1^2+k_4^2$.

Now, assume that $M$ is non-minimal and biharmonic. By the similar reasons to the case II we considered in the previous chapter, we assume  that $\nabla s_1$ is non-vanishing on $M$ and  proportional to the principal direction $e_4$ with $k_4=-s_1/2$. Therefore, we have
\begin{subequations}\label{Case2ShapeOperatorRes1All}
\begin{equation}\label{Case2ShapeOperatorRes1}
e_4=\frac{\nabla s_1}{\|\nabla s_1\|}, \quad k_4=-\kappa,\quad k_1=\kappa
\end{equation}
for a function $\kappa$ which imply
\begin{equation}\label{Case2ShapeOperatorRes1a}
e_4(\kappa)\neq0,\quad e_A(\kappa)=0,\quad A=1,2,3.
\end{equation}
\end{subequations}
We also put $\omega_{24}(e_1)=\tau_1$, $\omega_{14}(e_2)=\tau_2$ and $-\omega_{34}(e_3)=\tau_3$ and biharmonic equation \eqref{HSurfCondE51Case1e2} becomes
\begin{equation}\label{HSurfCondE51Case2e2v2}
e_4e_4(\kappa)+\left(\tau_1+\tau_2+\tau_3\right)e_4(\kappa)=4\kappa^3.
\end{equation}

Note that, $\langle X,e_4\rangle=0$ if and only if $X\kappa=0$. Therefore, \eqref{Case2ShapeOperatorRes1a} implies
\begin{equation}\label{HSurfCase2Res1Eq1} 
\omega_{A4}(e_B)=\omega_{B4}(e_A),\quad A,B=1,2,3
\end{equation}
which implies
\begin{equation}\label{HSurfCase2Res1} 
\tau_1=\tau_2=\tau
\end{equation}
for a function $\tau$.

We apply the Codazzi equation \eqref{MinkCodazzi} for $X=e_i,$ $Y=e_j$ and $Z=e_k$ for each triplet (i, j, k) in the set
$\{(1,4,4),(2,4,4),(3,4,4), (1,1,4), (2,2,4), (3,3,4),(1,2,4), (1,3,4), (1,2,2), (1,2,3)\}$
and combine equations obtained with \eqref{Case2ShapeOperatorRes1a} and \eqref{HSurfCase2Res1Eq1} to  get
\begin{subequations}\label{Cse2ConFormEq3ALL} 
\begin{eqnarray}
\label{Cse2ConFormEq3a} \omega_{14}(e_4)=\phi_2&=&0\\
\label{Cse2ConFormEq3b} \omega_{23}(e_2)=\omega_{23}(e_3)=\omega_{23}(e_4)&=&0,\\
\label{Cse2ConFormEq3c} \omega_{24}(e_2)=\omega_{24}(e_3)=\omega_{24}(e_4)&=&0,\\
\label{Cse2ConFormEq3d} \omega_{34}(e_2)=\omega_{34}(e_4)&=&0.
\end{eqnarray}
\end{subequations}
Moreover, by combaining \eqref{Cse2ConFormEq3ALL} with Codazzi equation \eqref{MinkCodazzi} for $X=e_1$, $Y=e_2$, $Z=e_4$, for $X=e_1$, $Y=e_3$, $Z=e_4$ and Gauss equation \eqref{MinkGaussEquation} for $X=Z=e_4,\ Y=W=e_3$ we obtain
\begin{subequations}\label{Cse2ConFormEq4ALL} 
\begin{eqnarray}
\label{Cse2ConFormEq4a} \tau_3&=&\tau,\\
\label{Cse2ConFormEq4b} e_4(\kappa)&=&-2\kappa\tau,\\
\label{Cse2ConFormEq4c} e_4(\tau)&=&\kappa^2-\tau^2.
\end{eqnarray}
\end{subequations}
By combaining \eqref{HSurfCase2Res1} and \eqref{Cse2ConFormEq4ALL}   with \eqref{HSurfCondE51Case2e2v2} we obtain $\kappa=0$. Hence, we have proved
\begin{theorem}
Let $M$ be a Lorentzian hypersurface in $\mathbb E^5_1$ with the shape operator given by \eqref{Case2ShapeOperator}. Then $M$ is biharmonic if and only if it is minimal.
\end{theorem}

\end{document}